\documentclass[11pt]{article}
\usepackage[totalwidth=13.0cm,totalheight=20.0cm]{geometry}
\usepackage{latexsym,amsthm,amsmath,amssymb,url}
\usepackage[ruled, linesnumbered]{algorithm2e}
\usepackage{mathrsfs}
\usepackage{tikz,authblk}
\usetikzlibrary{decorations.pathreplacing}

\newtheorem{theorem}{Theorem}[section]
\newtheorem{lemma}[theorem]{Lemma}
\newtheorem{corollary}[theorem]{Corollary}
\newtheorem{observation}[theorem]{Observation}

\newtheorem{conjecture}[theorem]{Conjecture}

\newtheorem{problem}[theorem]{Problem}

\begin{document}

	\title{Note on Disjoint Cycles in Multipartite Tournaments}
	\author{ Gregory Gutin\thanks{Department of Computer Science. Royal Holloway University of London. {\tt g.gutin@rhul.ac.uk}.}, \hspace{2mm}  Wei Li\thanks{School of Mathematics and Statistics, Northwestern Polytechnical University. {\tt liw@nwpu.edu.cn}.}, \hspace{2mm}  Shujing Wang\thanks {Corresponding author.  School of Mathematics and Statistics, Central China Normal University. {\tt wang06021@126.com}.}, \hspace{2mm}  Anders Yeo\thanks {Department of Mathematics and Computer Science, University of Southern Denmark and Department of Mathematics and Applied Mathematics, University of Johannesburg. {\tt yeo@imada.sdu.dk}.}, \hspace{2mm}  Yacong Zhou\thanks{Department of Computer Science. Royal Holloway University of London. {\tt Yacong.Zhou.2021@live.rhul.ac.uk}.} }
	
	\date{}
	
	\maketitle
		\begin{abstract}
In 1981, Bermond and Thomassen conjectured that for any positive integer $k$, every digraph with minimum out-degree at least $2k-1$ admits $k$ vertex-disjoint directed cycles. In this short paper, we verify the Bermond-Thomassen conjecture for triangle-free multipartite tournaments and 3-partite tournaments. Furthermore, we characterize 3-partite tournaments with minimum out-degree at least $2k-1$ ($k\geq 2$) such that in each set of $k$ vertex-disjoint directed cycles, every cycle has the same length.
	\end{abstract}
\vspace{2mm}



	\section{Introduction}
In this paper, digraphs do not have loops or parallel arcs.\footnote{Undefined terminology can be found later in this section or in the monograph \cite{Bang-JensenG18}.} For a digraph $D$, $V(D)$ ($A(D)$, resp.) are the vertex set (arc set, resp.) of $D$. 	 A digraph $D$ is {\em complete} if $xy$ and $yx$ are arcs for every pair $x,y$ of distinct vertices, i.e., $|A(D)|=|V(D)|(|V(D)|-1).$
The {\em minimum out-degree} of a vertex in $D$ is denoted by $\delta^+(D).$
An {\em orientation} of an undirected graph $G$ is a digraph obtained by replacing each edge $\{x,y\}$ of $G$ with arc $xy$ or $yx$. A {\em multipartite tournament} is an orientation of a complete multipartite graph. A multipartite tournament with $t\ge 2$ partite sets is a {\em $t$-partite tournament}.
Note that a tournament is a multipartite tournament in which every partite set is a singleton.

Bermond and Thomassen \cite{BermondT81} observed that a complete digraph $K^*_n$ on $2k-1$ vertices has $k-1$ vertex-disjoint cycles, but no $k$ such cycles. Note that $\delta^+(K^*_n)=2k-2$. Let $f(k)$ be the minimum integer such that if $\delta^+(D)\ge f(k)$ for a digraph $D$ then $D$  has $k$ vertex-disjoint cycles (if such an $f(k)$ does not exist then we set $f(k)=\infty$).
Bermond and Thomassen \cite{BermondT81} posed the  following:
\begin{conjecture}\label{conj:BT}
	For every $k\in \mathbb{Z}_{>0}$, $f(k)=2k-1.$	
	\end{conjecture}
	This conjecture is trivially true for $k=1$. Thomassen \cite{Thomassen83} showed it for $k=2.$
	The conjecture was proved for $k=3$ by Lichiardopol et al. \cite{LichiardopolPS09}; a shorter proof of this case is given by Bai et al. \cite{BaiLL15}.
	For every $k\in \mathbb{Z}_{>0}$, Thomassen \cite{Thomassen83} proved that $f(k)\le (k+1)!$. The upper bound on $f(k)$ was drastically improved to $64k$ by Alon \cite{Alon96} and improved further to $18k$ by Buci{\'c} \cite{Bucic18} who combined the probabilistic approach of Alon \cite{Alon96} and the result of Lichiardopol et al. \cite{LichiardopolPS09}.
	
In this paper, all cycles are directed. A cycle $C$ in a digraph $D$ is a {\em $p$-cycle} if $C$ has $p$ arcs. A 3-cycle is called a {\em triangle}. A digraph is {\em triangle-free} if it has no triangles.
	Conjecture \ref{conj:BT} was proved
	for tournaments by Bang-Jensen et al. \cite{Bang-JensenBT14} as  follows.
	\begin{theorem}\label{thm:jbj}
		Every tournament with $\delta^+(T)\ge 2k-1$ has $k$ vertex-disjoint triangles.
	\end{theorem}
	A natural question one would ask is whether we can prove this conjecture for some generalization of tournaments. The first generalization one may try is {\em extended tournaments}, where we blow up every vertex of a tournament to an arbitrary independent set. However, it is easy to see that the assertion that the conjecture holds for extended tournaments is  a simple corollary of Theorem \ref{thm:jbj}. In fact, we can replace every partite set of an extended tournament with an arbitrary transitive tournament on the same vertices and apply  Theorem \ref{thm:jbj} to the obtained tournament. Then, we are done by the fact that every triangle in the tournament is also a triangle in the extended tournament.
	
Extended tournaments are a special class of multipartite tournaments. Thus, let us consider multipartite tournaments. Bai et al. \cite{BaiLL15} proved Conjecture \ref{conj:BT} for bipartite tournaments.
	
\begin{theorem}\label{thm:bipartite tournament}
	Let $D$ be a bipartite tournament. If $\delta^+(D)\geq 2k-1$ then $D$ has at least $k$ vertex-disjoint $4$-cycles.
\end{theorem}	

In this paper, we prove that (i) a triangle-free multipartite tournament $D$ with $\delta^{+}(D)\geq 2k-1$ contains $k$ vertex-disjoint cycles, and
(ii) a $3$-partite tournament $D$ with $\delta^{+}(D)\geq 2k-1$ contains $k$ vertex-disjoint cycles.

Chen and Chang \cite{ChenChang} extended Theorems \ref{thm:jbj} and \ref{thm:bipartite tournament} where diversity of the lengths of the cycles is taken into consideration.
To state their results, we will use the following notation.
Let $\mathcal{C}$ be a set of $k$ vertex-disjoint cycles of $D$. Let $\kappa^k(\cal C)$ be the maximum number of cycles in $\cal C$ of distinct lengths. Let $$\kappa^k(D)=\max\{\kappa^k(\mathcal{C}): \mathcal{C}  \textit{~is a set of k vertex-disjoint directed cycles in}~D\}.$$

For tournaments, Chen and Chang \cite{ChenChang} proved the following:

\begin{theorem}
Let $D$ be a tournament with $\delta^+(D)\geq 2k-1$ and let $k\ge 3$.  Then $\kappa^k(D)\ge 2$.
\end{theorem}

For two subsets $X,Y$ of $V(D)$, let $A_D(X,Y)$ be the set of arcs from $X$ to $Y$ (we will omit the subscript if it is clear from the context).
For an integer $k\ge 2$ and positive integers $n_1,n_2,\dots, n_{2k}$, a digraph $D$ belongs to the digraph set $BT(n_1,n_2,\dots, n_{2k})$ if $D$ is a bipartite tournament with partite sets $X,Y$ such that (i) $X$ is the disjoint union of $X_1,X_2,\dots , X_{2k}$, where $|X_i|=n_i$ for every $i\in [2k]$, and  $Y$ is the disjoint union of $Y_1,Y_2,\dots , Y_{2k}$, where $|Y_i|=1$
for every $i\in [2k]$, and (ii) the arc set of $D$ is $$(\cup_{1\leq i \neq j\leq 2k} A(X_i,Y_j))\cup (\cup_{1\leq i \leq 2k} A(Y_i,X_i)).$$

Chen and Chang \cite{ChenChang} obtained the following characterization of strong bipartite tournaments $D$ with $\kappa^k(D)=1$.
	\begin{theorem}\label{A}
		Let $D$ be a strong bipartite tournament with $\delta^+(D)\geq 2k-1,$ where $k\geq 2$. If $\kappa^k(D)=1$, then $D$ is a member of $BT(n_1,n_2,\ldots, n_{2k})$, where $n_i\geq 2k-1$ for $i=1,\ldots, k$.
	\end{theorem}
In this paper, we obtain a characterization of 3-partite tournaments $D$ with $\delta^+(D)\geq 2k-1$ and $\kappa^k(D)=1$.

\vspace{2mm}

\noindent {\bf Additional Terminology and Notation}
Sometimes, if $xy\in A(D)$, we will write $x\rightarrow y$ when $D$ is known from the context. Let $C$ be a cycle and let $x,y$ be two vertices in the cycle. We write $C[x,y]$ to denote the $(x,y)$-path on $C$. Let $a_D(X,Y)=|A_D(X,Y)|$.

For  $x\in V(D)$, $N^+(x)$ ($N^-(x)$, resp.) denote the out-neighbourhood (in-neighbourhood, resp.) of $x$. We write $d^+(x)$ ($d^-(x)$, resp.) to denote the out-degree (in-degree, resp.) of $x$, i.e., $d^+(x)=|N^+(x)|$ ($d^-(x)=|N^-(x)|$, resp.). A vertex $v$ is a {\em sink} if $d^+(x)=0$.
For a subset $V'\subseteq V(D)$, we use $d^+_{V'}(x)$ to denote the number of out-neighbours of $x$ in $V'$, i.e., $d^+_{V'}(x)=|N^+(x)\cap V'|$.

	\section{Proofs}

	We need to use the following vertex-pancyclic-like result for multipartite tournaments obtained by Guo and Volkmann \cite{GV}. Initially, we obtained a long technical proof of Lemma \ref{lem:triangle-free} (and therefore Theorem \ref{thm:3-partite}) without using this result. However, with the help of this result, we can now give a short proof for this theorem.
	\begin{theorem}\cite{GV}\label{thm:v-pancyclic}
		If $D$ is a strong $t$-partite $(t\geq 3)$ tournament, then every partite set of $D$ has at least one vertex, which lies on an $m$-cycle for all $m\in\{3,4,\dots,t\}$.
	\end{theorem}
	
	\begin{lemma}\label{lem:triangle-free}
		If $D$ is a triangle-free multipartite tournament with $\delta^+(D)\geq 2k-1$, then $D$ has $k$ vertex-disjoint cycles.
	\end{lemma}
	\begin{proof}
Let $H$ be the terminal strong component of $D$. Clearly, $\delta^+(H)\geq 2k-1$. We claim that $H$ is a bipartite tournament. Otherwise, $H$ is a strong $t$-partite tournament with $t\geq 3$, and therefore, by Theorem \ref{thm:v-pancyclic}, $H$ has a triangle, which contradicts the fact that $D$ is triangle-free. Thus, we are done by applying Theorem \ref{thm:bipartite tournament} to $H$.
	\end{proof}
	
	\begin{theorem}\label{thm:3-partite}
		Let $D$ be a $3$-partite tournament and $\delta^{+}(D)\geq 2k-1$. Then $D$  contains $k$ vertex-disjoint cycles.
	\end{theorem}
	\begin{proof}
	We prove this by induction on $k$. When $k=1$, this result clearly holds. Assume that $k\geq 2$ and this theorem is true for $k-1$. If $D$ has a triangle $C$, then one can observe that $\delta^+(D-V(C))\geq 2k-1-2=2(k-1)-1$. By the induction hypothesis, $D-V(C)$ has $k-1$ vertex-disjoint cycles. Therefore, together with $C$, $D$ has at least $k$ vertex-disjoint cycles. If $D$ is triangle-free, then we are done by Lemma \ref{lem:triangle-free}.
	\end{proof}
	By a similar induction, we can prove the following result for all multipartite tournaments.
	\begin{corollary}
		Let $D$ be a multipartite tournament and $\delta^{+}(D)\geq 3k-2$. Then, $D$  contains $k$ vertex-disjoint cycles.
	\end{corollary}
	\begin{proof}
		We prove this by induction on $k$. When $k=1$, this result clearly holds. Assume that $k\geq 2$ and it is true for $k-1$. We may assume that $D$ has a triangle $C$ (for otherwise we are done by Lemma \ref{lem:triangle-free}). Then, $\delta^+(D-V(C))\geq 3k-2-3=3(k-1)-2$. And we are done by using the induction hypothesis.
	\end{proof}
	
	\begin{observation}\label{obs1}
		Let $D$ be a multipartite tournament and $S$ be its sink set, then all sinks lie in the same partite set and every vertex in $V(D)\setminus S$ can reach all vertices in $S$ by a path with length at most two.
	\end{observation}
	
	\begin{lemma}\label{lem:1}
		Let $D$ be a multipartite tournament and let $C$, $C'$ be two vertex-disjoint $3$-cycles in $D$. If $a(V(C),V(C'))>0$ and $a(V(C'),V(C))>0$, then $D[V(C)\cup V(C')]$ contains a cycle with length at least $4$.
	\end{lemma}
	\begin{proof}
		Let $V(C)=\{x,y,z\}$ and $V(C')=\{x',y',z'\}$. As $D$ is a multipartite tournament, $a(V(C),V(C'))+a(V(C'),V(C))\geq 6$. Thus, we may assume without loss of generality that $a(V(C),V(C'))\geq 3$ and $y'\to y$. Since $a(V(C),V(C'))\geq 3$, there is an arc from $\{x,z\}$ to $\{x',z'\}$. Assume w.l.o.g. that $x\to x'$. Then, $xx'C'[x',y']y'yC[y,x]x$ is a cycle with length at least 4.
	\end{proof}
	\begin{lemma}\label{lem:2}
		Let $D$ be a $3$-partite tournament and $C_1$, $C_2$ and $C_3$ be three vertex-disjoint $3$-cycles in $D$. If for all $i\in [3]$, $a(V(C_{i+1}),V(C_i))=0$ where $C_4=C_1$, then $D[V(C_1)\cup V(C_2)\cup V(C_3)]$ has two vertex-disjoint cycles with different lengths.
	\end{lemma}
	\begin{proof}
		Let $V(C_i)=\{x_i,y_i,z_i\}$. Assume w.l.o.g. that $x_i$'s ($y_i$'s, $z_i$'s, respectively) are in the same partite set and $x_1\to y_1\to z_1\to x_1$ (for otherwise we can exchange $y_i$'s and $z_i$'s). Thus, we have a 4-cycle $x_1y_1z_2y_3x_1$ and a 3-cycle $z_1y_2x_3z_1$.
	\end{proof}

	\begin{theorem}\label{thm}
		Every $3$-partite tournament $D$ with $\delta^+(D)\geq 2k-1$ satisfies that $\kappa^k(D)\geq 2$ unless $D$ is triangle-free.
	\end{theorem}
	
	\begin{proof}
		Denote the three partite sets of $D$ by $V_1$, $V_2$, and $V_3$. Assume to the contrary that $D$ contains a triangle and  $\kappa^{k}(D)=1$. By the proof of Theorem \ref{thm:3-partite}, it suffice to consider the case when we have collection of $k$ vertex-disjoint 3-cycles $\mathcal{C}=\{C_1,\ldots, C_k\}$ where $C_i=x_i y_i z_i x_i$ for all $i\in [k]$. Let $V( \mathcal{C})=\cup_{i=1}^k V(C_i)$, $U=D-V(\mathcal{C})$ and $S$ the set of sinks in $U$ (i.e., $S=\{x\in V(U):d^+_{V(U)}(x)=0\}$). We will consider the following two cases.\\
		
		\textbf{Case 1:} $S\neq \emptyset$.
		
		By Observation \ref{obs1}, we may assume without loss of generality that $S\subseteq V_3$ and $z_i\in V_3$ for all $i\in [k]$. Note that for every $s\in S$, we have that
		$d^+_{V(\cal C)}(s)=d^+(s)\geq 2k-1$. Since $(N^+(s)\cup N^-(s))\cap V(\mathcal{C})\subseteq V(\mathcal{C})\cap (V_1\cup V_2)$, $d_{V(\mathcal{C})}^{+}(s)+d_{V(\mathcal{C})}^{-}(s)= 2k$. Hence, for any $s\in S$, we have that
		\begin{equation}\label{eq1}
			d^-_{V(\mathcal{C})} (s)\leq 1.
		\end{equation}
		{\bf Claim 1.A.} For any $i\in [k]$, $d^+_{V(U)}(z_i)=0$ and $N^+(z_i)=\{x_1,y_1,\ldots, x_k, y_k\}\setminus \{y_i\}$.
		\begin{proof}
Assume that there exists $u\in V(U)$ such that $z_iu\in A(D)$. As $z_i\in V_3$, we have that $u\notin V_3$ and therefore $us\in A(D)$ for all $s\in S$. Note that by \eqref{eq1}, we have that either $sx_i\in A(D)$ or $sy_i\in A(D)$. Thus, either $z_i u s x_i y_i z_i$ or $z_i us y_i z_i$ is a cycle of $D$. This cycle together with cycles in $\mathcal{C}\setminus \{C_i\}$ forms a collection of $k$ vertex-disjoint cycles that have $k-1$ $3$-cycles and one cycle of length at least 4, a contradiction. Thus, $N^+(z_i)\subseteq V({\mathcal{C}})$. Since $d^+(z_i)\geq 2k-1$, we have $N^+(z_i)=\{x_1,y_1,\ldots, x_k, y_k\}\setminus \{y_i\}$.
		\end{proof}
		\noindent{\bf Claim 1.B.} For  $i\in [k], S\subseteq N^+(y_i)$.
		
		\begin{proof}
		
			For any $i\in [k]$, by Claim 1.A, we have that $d^+_{V(\mathcal{C})}(x_i)\leq k$ and therefore $d^+_{V(U)}(x_i)\geq k-1>0$, which implies that $x_i$ has an out-neighbour in $V(U)$. Let $u\in V(U)$ be the out-neighbour of $x_i$. If $u\in S$, then, by \eqref{eq1}, we have that $uy_i\in A(D)$. Thus we have that $C_i':=x_iuy_iz_ix_i$ is a cycle of length 4 and $\mathcal{C}':=(\mathcal{C}\setminus \{C_i\})\cup \{C_i'\}$ is a collection of $k$ vertex-disjoint cycles with $\kappa^k(\mathcal{C}')=2$, a contradiction. Thus, $u\in V(U)\setminus S$.

			Now, we assume to the contrary that there is a vertex $s\in S$ such that $sy_i\in A(D)$. By observation \ref{obs1}, there is an $(u,s)$-path with length at most 2 in $U$. The existence of this path together with $s\to y_i\to z_i\to x_i\to u$ implies a cycle with length at least 4 in $D[V(U)\cup V(C_{i})]$ which together with cycles in $\mathcal{C}\setminus \{C_i\}$ implies that $\kappa^k(D)\geq 2$, a contradiction.
		\end{proof}
		
		By Claim 1.B, we have that $d^-_{V(\mathcal{C})} (s)\geq k\geq2$ for all $s\in S$, a contradiction to \eqref{eq1}. This completes the proof of Case 1.\\

		{\bf Case 2.} $S=\emptyset$.
		
	 Let $D'$ be the terminal strong component of $U$. Note that $\delta^+(D')=1$ for otherwise $\delta^+(D')\geq 2$ and therefore $D'$ contains a cycle of length at least 4.  Let $C_{k+1}=x_{k+1}y_{k+1}z_{k+1}x_{k+1}$ be a cycle in $D'$ that includes a vertex of the minimum out-degree. Assume without loss of generality that $d^+_{V(U)}(x_{k+1})=1$.\\
	 {\bf Claim 2.A.} For  $i,j \in [k+1]$, all arcs between $V(C_i)$ and $V(C_j)$ go in the same direction.
		\begin{proof}
			Assume to the contrary that $a(V(C_i), V(C_j))>0$ and $a(V(C_j), V(C_i))>0$. By Lemma \ref{lem:1}, there is a cycle of  length at least 4 in $D[V(C_i)\cup V(C_j)]$ and therefore we are done by the fact that this new cycle together with cycles $(\mathcal{C}\cup \{C_{k+1}\})\setminus \{C_i,C_j\}$ forms a set of $k$ cycles that have $k-1$ 3-cycles and one cycle with length at least 4, a contradiction.
		\end{proof}

		Construct a digraph $T$ with $k+1$ vertices $c_1$, $\dots$, $c_{k+1}$ where $c_i\to c_j$ if and only if $a(V(C_i),V(C_j))>0$. By Claim 2.A, $T$ is a tournament.\\
		{\bf Claim 2.B.} $T$ is a transitive tournament.
		\begin{proof}
			Suppose it is not. Then by Theorem \ref{thm:v-pancyclic}, $T$ has a 3-cycle $c_ic_jc_hc_i$, which, by Claim 2.A, implies that $a(V(C_j), V(C_i))=a(V(C_h), V(C_j))=a(V(C_i), V(C_h))=0$. By Lemma \ref{lem:2}, there are two vertex-disjoint cycles with different lengths in $D[V(C_i)\cup V(C_j) \cup V(C_h)]$. These two cycles, together with cycles in $(\mathcal{C}\cup \{C_{k+1}\})\setminus \{C_i,C_j,C_h\}$, forms $k$ cycles with at least two of them having different lengths, a contradiction.
		\end{proof}
		
		By Claim 2.B, we assume w.l.o.g. that $c_1$, $c_2$, $\dots$, $c_k$ is an acyclic ordering in $T-c_{k+1}$. Since $d^+_{V(U)}(x_{k+1})=1$, $d^+_{V(\mathcal{C})}(x_{k+1})\geq 2(k-1)$. Thus, $d^+_T(c_{k+1})\geq k-1$ and therefore either $c_{k+1}$, $c_1$, $c_2$, $\dots$, $c_k$ or $c_1$, $c_{k+1}$, $c_2$, $\dots$, $c_k$ is an acyclic ordering of $T$. In both case we have $c_{k+1}\to c_k$. Recall that $C_k=x_ky_kz_kx_k$. Since $c_k$ is the sink in $T$, $d^+_{V(\mathcal{C})}(x_k)=d^+_{V(\mathcal{C})}(y_k)=d^+_{V(\mathcal{C})}(z_k)=1$. In particular, $d^+_{V(U)}(x_k)\geq 2(k-1)>0$. Therefore, as $c_{k+1}\to c_k$, $x_k$ has an out-neighbour $u$ in $V(U)\setminus V(C_{k+1})$. On the other hand, since $D$ is a multipartite tournament, $D'$ is the unique terminal component of $U$. Thus, there is an $(u,x_{k+1})$-path $P$ in $U$. Since $c_{k+1}\to c_{k}$, $x_{k+1}$ dominates $y_{k}$ or $z_{k}$ (or both). Thus, $x_kuPx_{k+1}y_{k}C_k[y_{k},x_{k}]x_k$ or $x_kuPx_{k+1}z_{k}C_k[z_{k},x_{k}]x_k$ is a cycle with length at least 4 in $D[V(U)\cup V(C_{k})\cup V(C_{k+1})]$ which together with $k-1$ 3-cycles in $\mathcal{C}\setminus \{C_k\}$ implies that $\kappa^k(D)\geq 2$, a contradiction. This completes the proof.
	\end{proof}
	\begin{theorem}
		Let $D$ be a $3$-partite tournament with $\delta^+(D)\geq 2k-1$. Then $\kappa^k(D)=1$ if and only if the only nontrivial strong component is the terminal strong component $D'$ and $D'$ is a member of $BT(n_1,\ldots, n_{2k})$, where $n_i\geq 2k-1$ for $i=1,\ldots,2k$.
	\end{theorem}
	
	\begin{proof}
		Let $D$ be a 3-partite tournament with $\delta^+(D)\geq 2k-1$ and $\kappa^k(D)=1$. By Theorem \ref{thm}, $D$ is triangle-free. Let $D_1, \ldots, D_s$ be an acyclic ordering of the strong components of $D$. Note that $\delta^+(D_s)\geq \delta^+(D)\geq 2k-1$. Then by Theorem \ref{A} and Theorem \ref{thm:v-pancyclic}, $D_s$ is a member of $BT(n_1,\ldots, n_{2k})$, where $n_i\geq 2k-1$ for every $i\in [2k]$. One can observe that there exists a collection of $k-1$ vertex-disjoint cycles $\mathcal{C}$ in $D_s$, where $k-2$ of them are 4-cycles and one of them is 6-cycle. For $i\not= s$, if $D_i$ is strong, then there exists a cycle $C'$ in $D_i$, and $\kappa^k(\mathcal{C}\cup \{C'\})\geq 2$, a contradiction. Thus, for every $i\neq s$, $D_i$ is a single vertex. This completes the proof.
	\end{proof}

		\section{Discussion}\label{sec:disc}
	
We proved the Bermond-Thomassen conjecture for triangle-free multipartite tournaments and 3-partite tournaments. We also characterized 3-partite tournaments with minimum out-degree at least $2k-1$ ($k\geq 2$) in which every $k$ vertex-disjoint directed cycles have the same length. This conjecture seems to be very difficult even for multipartite tournaments. In fact, we were unable to verify the conjecture even for 4-partite tournaments, in general. It would be interesting to see a proof of this conjecture for multipartite tournaments.

		\begin{problem}
		Verify that every multipartite tournament $D$ with $\delta^+(D)\geq 2k-1$ has $k$ vertex-disjoint cycles.
	\end{problem}
	Perhaps, it'd be useful to consider the following two subproblems as a starting point. An undirected graph $G=(V,E)$ is a {\em split graph} if $V$ can be partitioned into two non-empty subsets $X$ and $Y$ such that $G[X]$ is a complete graph and $G[Y]$ has no edges. A split graph $G=(V,E)$ with partition $X,Y$ of $V$ is a {\em complete split graph} if every pair of vertices $x\in X$ and $y\in Y$ is adjacent.
A {\em complete split oriented graph} is an orientation of a complete split graph. Note that a complete split digraph is a multipartite tournament.
	\begin{problem}
		Verify that every complete split oriented graph $D$ with $\delta^+(D)\geq 2k-1$ has $k$ vertex-disjoint cycles.
	\end{problem}
		\begin{problem}
		Verify that every $4$-partite tournament $D$ with $\delta^+(D)\geq 2k-1$ has $k$ vertex-disjoint cycles.
	\end{problem}

	\paragraph{Acknowledgement} This research of Li was supported by China Scholarship Council (CSC) and the Natural Science Basic Research Planin Shaanxi Province of China (Grant No. 2020JM-133). This research of Wang and Zhou was supported by China Scholarship Council (CSC).
This research of Yeo was supported by the danish research council under grant number DFF-7014-00037B.

	\end{document}